\documentclass[11pt,reqno]{amsart}
\usepackage[a4paper,textwidth=154mm,hcentering,top=28mm,bottom=28mm]{geometry}
\usepackage{amsmath,amssymb,amsthm,booktabs,hyperref,needspace}
\hypersetup{colorlinks=true,urlcolor=blue,linkcolor=blue,citecolor=blue,filecolor=blue,pdfborder={0 0 0}}
\numberwithin{equation}{section}
\newtheorem{theorem}{Theorem}[section]
\newtheorem{lemma}[theorem]{Lemma}
\newtheorem{proposition}[theorem]{Proposition}

\theoremstyle{definition}

\newtheorem{problem}{Problem}
\theoremstyle{remark}

\title[Optimal card guessing under nonincreasing valley weights]{Optimal complete-feedback card guessing under nonincreasing valley weights: A proof of the Diaconis--Fulman--Holmes shelf-shuffling conjecture}
\author{Congyi Luo}
\address{School of Data Science, Fudan University, Shanghai 200433, China}
\email{cyluo24@m.fudan.edu.cn}
\date{}
\begin{document}
\begin{abstract}
Let a random permutation $W\in S_n$ represent the order of $n$ distinct cards. After each guess, the player learns the actual card, and the objective is to maximize the expected total number of correct guesses. Diaconis, Fulman, and Holmes conjectured that, after a single uniform, unbiased shelf shuffle, the direction-tracking strategy is optimal: first guess the smallest card, and thereafter choose the nearest remaining card in the direction of the most recent rise or fall. We prove a more general theorem: if
\[
\mathbb P(W=w)=\frac{f(v(w))}{\sum_{\sigma\in S_n}f(v(\sigma))},
\]
where $v(w)$ is the number of interior valleys of the permutation, $f$ is nonnegative and nonincreasing, and the denominator is positive, then direction tracking maximizes the conditional probability of correctly guessing the next card after every history of positive probability. It therefore maximizes the expected total number of correct guesses among all complete-feedback strategies. Here an interior valley is a position $2\le i\le n-1$ satisfying $w_{i-1}>w_i<w_{i+1}$. Combined with the known shelf-shuffling probability formula, this theorem proves the conjecture for every deck size and number of shelves. The key to the proof is a swap of adjacent remaining values that preserves the revealed prefix: starting from the preferred candidate, the swap leaves the valley count unchanged or increases it by one, thereby turning monotonicity of the weights into a conditional probability ordering after every history.
\end{abstract}
\maketitle
\section{Introduction}
The basic question in complete-feedback card guessing is this: given a distribution on deck orders, how should one guess to maximize the expected total number of correct guesses when the actual card is revealed after every guess? Diaconis and Graham~\cite[Section 2]{DG} studied strategies and rewards for this problem with a uniformly shuffled deck, allowing repeated card values. For a uniform random permutation of distinct cards, all remaining cards are equally likely at every round. For a permutation obtained by shuffling an ordered deck once, the revealed cards also carry information about the order of those that remain. We ask which properties of a permutation distribution guarantee the optimality of a simple complete-feedback strategy, taking a shelf-shuffling conjecture as our starting point.

Diaconis, Fulman, and Holmes~\cite[Section 5.1]{DFH} proposed the direction-tracking strategy: first guess the smallest card, and then guess the nearest remaining card in the direction of the most recently observed rise or fall. They conjectured that this rule is optimal under complete feedback for uniform, unbiased shelf shuffling. The increasing and decreasing orders retained within shelves suggest why the strategy might work. However, the player does not know the shelf boundaries, and an observed rise or fall does not uniquely determine the shelf from which the next card will come. The proof must compare all remaining candidates \emph{after each given history}, rather than merely compare their unconditional probabilities.

Our main result is an optimality theorem for complete-feedback card guessing under nonincreasing valley weights (Theorem~\ref{thm:general}). An interior position is a valley if its value is smaller than both neighboring values. Whenever the probability of an individual permutation depends only on its valley count and is nonincreasing in that count, direction tracking selects a remaining card of largest conditional probability after every history of positive probability. It therefore maximizes the expected total number of correct guesses among all complete-feedback strategies, including those using independent auxiliary randomness. This gives a sufficient condition for optimality in other permutation models: it suffices to establish the form and monotonicity of the permutation weights, without computing the conditional distribution after every history.

The key to the general theorem is to turn an unconditional comparison of permutation weights into a comparison of candidates after a fixed history. We swap two adjacent values in the remaining set, preserving the entire revealed prefix, and show that, starting from the preferred candidate, the swap leaves the valley count unchanged or increases it by one. Nonincreasing weights then give a comparison for each permutation. Summing over completions yields the conditional probability ordering after every history. Sections~\ref{sec:swap}--\ref{sec:optimality} develop this argument without using the particular form of any shuffling probability formula.

Diaconis, Fulman, and Holmes~\cite[Theorem 3.1 and Corollary 3.3]{DFH} proved that the probability of an individual shelf-shuffling output depends only on its valley count and is nonincreasing in that count. The general theorem therefore proves their conjecture for every positive integer deck size and number of shelves. We first define the shuffling model, the strategy, and the competing strategies, and state this shelf-shuffling result in Theorem~\ref{thm:shelf}. We then give the general theorem and its proof, before completing the application to shelf shuffling in Section~\ref{sec:application}.
\subsection{The model, the strategy, and the main theorem}\label{sec:model}
Shelf shuffling preserves part of the pattern of rises and falls in the original deck. The direction-tracking strategy uses the most recent rise or fall to choose the next card. We specify this rule for every history and state its optimality. Throughout, optimality means maximizing the expected total number of correct guesses, not making the most correct guesses for every realized deck order.

Let $n,m$ be positive integers, and arrange $n$ distinct cards in the order $1,\ldots,n$ from top to bottom. The machine successively removes the bottom card, thus processing the cards in the order $n,n-1,\ldots,1$. Each card independently chooses one of the $m$ shelves uniformly and is then independently placed on the top or bottom of that shelf's pile, with probability $1/2$ each. Finally, the piles are concatenated in a fixed shelf order to give the output permutation $W=(W_1,\ldots,W_n)$, read from top to bottom. This is the shelf-shuffling model of~\cite[Section 3.1]{DFH}.

The player does not observe the shuffle. In round $j$, the player guesses $W_j$ using the actual cards previously revealed, and then sees $W_j$ whether or not the guess was correct. This information regime is called \emph{complete feedback}. Write $[n]=\{1,\ldots,n\}$. After a history $h=(w_1,\ldots,w_k)$ of length $k<n$, the set of remaining cards is
\[
 R=[n]\setminus\{w_1,\ldots,w_k\}.
\]
For the empty history, take $R=[n]$. If $k\ge1$, let $L=w_k$ be the last revealed card and partition the remaining cards into two sides:
\[
 R_- =\{r\in R:r<L\},\qquad R_+=\{r\in R:r>L\}.
\]
When $k=1$, the current direction is defined to be upward. When $k\ge2$, it is upward if $w_{k-1}<w_k$ and downward if $w_{k-1}>w_k$.

Write $G(h)$ for the strategy's choice in round $k+1$. For the empty history, $G(\varnothing)=1$. For a nonempty history, the direction-tracking rule is
\begin{equation}\label{eq:strategy}
G(h)=
\begin{cases}
\min R_+, & \text{upward, }R_+\ne\varnothing,\\
\max R_-, & \text{upward, }R_+=\varnothing,\\
\max R_-, & \text{downward, }R_-\ne\varnothing,\\
\min R_+, & \text{downward, }R_-=\varnothing.
\end{cases}
\end{equation}
Since $R\ne\varnothing$, each extremum used here is defined. The rule selects the nearest remaining card in the current direction, turning to the other side if no card remains in that direction. Only the actual revealed cards update the direction; whether a guess was correct plays no part in this update. Formula~\eqref{eq:strategy} makes the endpoint convention in the original description explicit.

Let $A_j$ be the guess made by strategy $A$ in round $j$, and define the total number of correct guesses by
\[
 C_A=\sum_{j=1}^n\mathbf1_{\{A_j=W_j\}},
\]
where $\mathbf1_E$ denotes the indicator of an event $E$. Let $\mathcal A_n$ be the class of all admissible complete-feedback strategies, including deterministic strategies and randomized strategies using auxiliary randomness independent of $W$. A card that has been guessed may be guessed again. A card that has been \emph{revealed}, however, cannot reappear, so guessing it has success probability zero.

The next theorem states both the original conjecture and the stronger conclusion that holds after every history. Write $\mathbb P_{n,m}$ and $\mathbb E_{n,m}$ for probability and expectation in the model with $n$ cards and $m$ shelves.
\begin{theorem}\label{thm:shelf}
For all integers $n,m\ge1$, the direction-tracking strategy satisfies
\begin{equation}\label{eq:main-optimality}
\mathbb E_{n,m}C_G
=\max_{A\in\mathcal A_n}\mathbb E_{n,m}C_A.
\end{equation}
More precisely, for every $0\le k<n$ and every history $h=(w_1,\ldots,w_k)$ of positive probability, with remaining set $R$, we have
\begin{equation}\label{eq:main-local}
G(h)\in\operatorname*{arg\,max}_{c\in R}
\mathbb P_{n,m}(W_{k+1}=c\mid W_1=w_1,\ldots,W_k=w_k).
\end{equation}
For the empty history, the probability on the right is the unconditional probability of the first card.
\end{theorem}
Formula~\eqref{eq:main-local} allows ties and does not assert uniqueness of the optimal strategy. It also shows that $G$ is optimal simultaneously for every $m$, although its rule does not depend on $m$. The maximum in~\eqref{eq:main-optimality} is attained by the explicit strategy $G$; it is not merely a supremum.
\subsection{Previous results and the proof outline}
Clay~\cite[Theorems 1.4 and 1.5]{Clay} studied the optimal strategy and its expected reward in the one-shelf case. Clay, Kuba, and Tripathi~\cite{CKT} determined the optimal strategy for biased one-shelf shuffling and studied its reward distribution, limit laws, and phase transitions as the bias approaches the endpoints. Their Section 4.1 discusses the complete-feedback problem for several shelves separately from the one-shelf results. We prove the unbiased shelf-shuffling optimality conjecture posed by Diaconis, Fulman, and Holmes in 2013~\cite[Section 5.1]{DFH}, for every positive integer deck size and number of shelves.

Results for related models help identify the difficulty. Bayer and Diaconis~\cite[Section 5.1]{BD} used complete-feedback card guessing to measure the order information retained after riffle shuffling. Liu~\cite{Liu} determined the optimal strategy after one riffle shuffle and obtained an asymptotic formula for its expected reward. Krityakierne and Thanatipanonda~\cite{KT} studied higher moments and gave a counterexample showing that the strategy for one riffle shuffle does not extend directly to repeated riffle shuffles. Thus, even with the same feedback, increasing the number of shuffles can change the ordering of conditional probabilities. Shelf shuffling also allows increasing and decreasing segments to alternate, so the riffle-shuffling strategy cannot simply be transferred. Our proof compares the weights of complete permutations with a fixed prefix, without requiring the player to identify the segment containing the current card.

The feedback regime also matters within the same shuffling model. Tripathi~\cite{TripathiPosition} determined the spectrum of the one-shelf position matrix and studied the optimal expected reward without feedback. The position matrix records the marginal probability of each card at each position. Complete feedback requires candidates to be compared anew after an arbitrary revealed prefix. The central issue here is therefore to turn an unconditional property of permutation weights into all of these conditional comparisons, rather than merely determine the position matrix.

The starting point is the permutation probability formula in~\cite[Theorem 3.1 and Corollary 3.3]{DFH}: the probability of an individual output permutation depends only on its number of interior valleys and is nonincreasing in that number. An interior position is a valley if its value is smaller than both neighboring values. Swapping two adjacent \emph{remaining values} preserves the revealed history and their comparisons with every other remaining value. The valley count can therefore change only near the boundary between the prefix and suffix. We show that, starting from the candidate preferred by the strategy, the swap either leaves the valley count unchanged or increases it by one. Nonincreasing weights then give a termwise comparison between the two sets of completions.

This argument does not use the particular form of the shuffling probability formula. Section~\ref{sec:valleys} states the general theorem for distributions with nonincreasing valley weights. Section~\ref{sec:swap} proves the swap lemma, and Section~\ref{sec:optimality} derives the conditional probability ordering and general optimality. Section~\ref{sec:application} verifies the hypotheses for shelf shuffling, completing the proof of Theorem~\ref{thm:shelf}. Section~\ref{sec:conclusion} formulates the reward questions that remain after optimality has been settled. The appendices discuss equality conditions for optimal choices, optimality when the number of shelves is unknown, and connections with complete-feedback guessing in uniformly shuffled decks.
\section{A general theorem for nonincreasing valley weights}\label{sec:valleys}
The main theorem will follow from a sufficient condition that does not refer to the shuffling procedure: the probability of an individual permutation depends only on its number of valleys and is nonincreasing in that number. Under this condition, direction tracking selects a most probable card after every history. We state this general result, which can be used for other permutation models. Section~\ref{sec:application} will then deduce the original conjecture from the known shuffling probability formula.

Let $S_n$ denote the permutations of $[n]$. An interior position $i$ of $w\in S_n$ is a \emph{valley} if $w_{i-1}>w_i<w_{i+1}$. Write
\[
 v(w)=\#\{i:2\le i\le n-1,\ w_{i-1}>w_i<w_{i+1}\}.
\]
For example, $(1,3,2,4)$ has one valley, whereas $(1,3,4,2)$ has none. The first and last positions are not counted. Two adjacent positions cannot both be valleys, so
\[
 0\le v(w)\le K,\qquad K=\left\lfloor\frac{n-1}{2}\right\rfloor.
\]
Indeed, for $n\ge3$, pair the interior positions $2,\ldots,n-1$ consecutively. Each pair contains at most one valley; if the last position is unpaired, allow one more. This gives the bound $\lceil(n-2)/2\rceil=K$. For $n=1,2$, there are no valleys, and the same bound holds.

Let $f:\{0,\ldots,K\}\to[0,\infty)$ be nonincreasing, and suppose its total weight over all permutations,
\[
 Z=\sum_{w\in S_n}f(v(w)),
\]
is strictly positive. Define a permutation distribution by
\begin{equation}\label{eq:general-law}
 \mathbb P_f(W=w)=\frac{f(v(w))}{Z},\qquad w\in S_n.
\end{equation}
The monotonicity assumption concerns the weight of an \emph{individual permutation}. The number of permutations with a given valley count can vary, so the probability of the event $\{v(W)=j\}$ need not be nonincreasing in $j$.

\begin{theorem}\label{thm:general}
Under~\eqref{eq:general-law}, let $p_h(c)$ be the conditional probability that the next card is $c$ after a history $h$ of positive probability, and let $R$ be the remaining set. Then
\[
G(h)\in\operatorname*{arg\,max}_{c\in R}p_h(c),
\qquad
\mathbb E_f C_G=\max_{A\in\mathcal A_n}\mathbb E_f C_A.
\]
\end{theorem}
\section{Swap comparisons with a fixed prefix}\label{sec:swap}
After fixing a revealed history, comparing two possible next cards requires comparing the total weights of two sets of complete permutations. We construct a bijection between these sets. The key is to swap only \emph{two adjacent values in the remaining set}: their comparisons with every other remaining card agree, so changes in the valley count can be determined locally at the boundary between prefix and suffix.

Fix a history $h$ of length $k<n$ and its remaining set $R$. To include the first round in the same argument, prepend a virtual value $0$ to the entire permutation. This does not change the valley count: the original first card now has the smaller left neighbor $0$ and still cannot be a valley, while comparisons at all other real positions are unchanged. The virtual value itself is not counted as a valley. For the empty history, take the boundary value to be $L=0$ and the direction to be upward; for a nonempty history, retain $L=w_k$. Thus, when only one card has been revealed, the direction from $0$ to $L$ is also upward.

Take adjacent values $a<b$ in $R$. Here adjacency means that no value in $R$ lies strictly between $a$ and $b$; revealed cards may lie in this interval. Let $\tau$ swap $a$ and $b$ in the complete permutation and fix all other cards. Since neither value has been revealed, $\tau$ preserves the entire history and bijectively maps permutations whose next card is $a$ to those whose next card is $b$. Its inverse is $\tau$ itself.

The following lemma identifies the direction of the swap that cannot decrease the valley count. If both cards lie on the same side of $L$, prefer the nearer one; if they lie on opposite sides, prefer the one in the current direction.

\Needspace{13\baselineskip}
\begin{lemma}\label{lem:swap}
Choose $c$ from $a,b$ by the following rule, and write $d$ for the other card:
\[
\begin{array}{c|c}
\text{Condition}&c\\\hline
L<a<b&a\\
a<b<L&b\\
a<L<b,\ \text{upward}&b\\
a<L<b,\ \text{downward}&a
\end{array}
\]
For every complete permutation $w$ beginning with $h,c$,
\begin{equation}\label{eq:swap}
 0\le v(\tau w)-v(w)\le1.
\end{equation}
\end{lemma}
\begin{proof}
Every suffix value $x$ other than $a,b$ satisfies $x<a$ or $x>b$, so its comparison with $a$ agrees with its comparison with $b$. A valley is determined by two adjacent comparisons. Hence every adjacent comparison within the suffix that does not directly join $a$ and $b$ is unchanged by the swap. Comparisons within the prefix are also unchanged. There is only one comparison across the prefix boundary: that between $L$ and the next card.

If the positions of $a,b$ in the suffix are not adjacent, there is no direct comparison between them, so only the comparison across the prefix boundary can change. The only positions whose valley status can change are therefore $L$ and the first suffix position. If their positions are adjacent, one is the first suffix position and the other must be the second. In that case, only $L$ and the first two suffix positions can change valley status. This explains why it suffices below to consider these positions, without distinguishing the possible more distant positions of the other card.

\emph{Case 1: $L<a<b$.}
The preferred value is $a$. If the positions of $a,b$ are not adjacent, the boundary comparison is upward before and after the swap. All adjacent comparisons, and hence the valley count, are unchanged. If their positions are adjacent, compare
\[
 (L,a,b,z,\ldots)\quad\text{and}\quad(L,b,a,z,\ldots),
\]
where $z$ denotes the third remaining card when it exists; otherwise $z$ and all later entries are omitted. In both permutations, the right neighbor of $L$ is larger, so the valley status of $L$ is unchanged. Originally, both $a$ and $b$ have smaller left neighbors and are not valleys. After the swap, $b$ is still not a valley, while $a$ is a valley exactly when $z$ exists and $z>b$. If $z<a$, then $a$ has a smaller right neighbor; if $z$ does not exist, then $a$ is the final position. Neither situation adds a valley.

\emph{Case 2: $a<b<L$.}
The preferred value is $b$. If the positions of $a,b$ are not adjacent, the boundary comparison is downward before and after the swap, so again the valley count is unchanged. If their positions are adjacent, compare
\[
 (L,b,a,z,\ldots)\quad\text{and}\quad(L,a,b,z,\ldots).
\]
In both permutations, $L$ has a smaller right neighbor and is not a valley. Originally, $b$ has a smaller right neighbor and is not a valley, while $a$ is a valley exactly when $z$ exists and $z>b$. After the swap, both neighbors of $a$ are larger, so $a$ is always a valley; $b$ has a smaller left neighbor and is not a valley. The valley count therefore increases by $1-\mathbf1_{\{z\text{ exists and }z>b\}}$, which is zero or one.

\emph{Case 3: $a<L<b$.}
The history is now nonempty. Let $D$ equal $1$ if the current direction is downward and $0$ if it is upward. Let $w_b$ be any permutation beginning with $h,b$, and set $w_a=\tau w_b$. If the positions of $a,b$ are not adjacent, let $z$ be the second suffix card. It necessarily exists and satisfies $z<a$ or $z>b$. In $w_b$, the position of $L$ is a valley exactly when $D=1$, and the first suffix card $b$ has a smaller left neighbor and is not a valley. In $w_a$, the position of $L$ has a smaller right neighbor and is not a valley, while the first suffix card $a$ is a valley exactly when $z>b$. All other valley statuses are unchanged. Thus
\begin{equation}\label{eq:cross-separated}
 v(w_a)-v(w_b)=\mathbf1_{\{z>b\}}-D.
\end{equation}
If the positions of $a,b$ are adjacent, compare
\[
 (L,b,a,z,\ldots)\quad\text{and}\quad(L,a,b,z,\ldots).
\]
In $w_b$, the position of $L$ contributes $D$ valleys, $b$ is not a valley, and $a$ contributes $\mathbf1_{\{z\text{ exists and }z>b\}}$ valleys. In $w_a$, the position of $L$ is not a valley, $a$ is always a valley, and $b$ is not a valley. Consequently,
\begin{equation}\label{eq:cross-adjacent}
 v(w_a)-v(w_b)=1-D-\mathbf1_{\{z\text{ exists and }z>b\}}.
\end{equation}
When $L$ is the first real card, its left neighbor is the virtual value $0$ and $D=0$, so the count at $L$ remains valid. When the suffix consists only of $a,b$, the indicator in~\eqref{eq:cross-adjacent} is zero, which also covers the final-position case.

If the direction is upward, then $D=0$ and both right-hand sides are zero or one. The preferred value is $b$, and $v(w_a)\ge v(w_b)$. If the direction is downward, then $D=1$ and both right-hand sides are zero or minus one. The preferred value is $a$, and $v(w_b)\ge v(w_a)$. These statements give~\eqref{eq:swap}.
\end{proof}

The lemma compares individual permutations. Since $f$ is nonincreasing, swapping from the preferred value cannot increase the permutation weight. Summing over all completions of the fixed prefix will therefore compare the conditional probabilities of the next card.
\section{Conditional probability ordering and optimality}\label{sec:optimality}
The swap lemma does more than identify an optimal candidate. It orders conditional probabilities on each side of the boundary: on the same side, a remaining value closer to the last revealed card is at least as favorable; between the two nearest cards on opposite sides, the current direction determines the comparison. We first prove this ordering and then use complete feedback to pass from individual rounds to the total reward.

Work under the distribution~\eqref{eq:general-law}. Fix a history $h=(w_1,\ldots,w_k)$ of positive probability and write
\[
 p_h(c)=\mathbb P_f(W_{k+1}=c\mid W_1=w_1,\ldots,W_k=w_k),
 \qquad c\in R.
\]
For the empty history, $p_h(c)=\mathbb P_f(W_1=c)$. The boundary value is $L=w_k$ for a nonempty history and $L=0$ for the empty history.

\begin{proposition}\label{prop:ordering}
List the remaining cards on the two sides of $L$ as
\[
 u_1<\cdots<u_t<L<r_1<\cdots<r_s.
\]
On each nonempty side, respectively,
\begin{equation}\label{eq:side-order}
 p_h(u_1)\le\cdots\le p_h(u_t),\qquad
 p_h(r_1)\ge\cdots\ge p_h(r_s).
\end{equation}
If both sides are nonempty, then $p_h(r_1)\ge p_h(u_t)$ when the direction is upward, and $p_h(u_t)\ge p_h(r_1)$ when it is downward. In particular, $G(h)$ always maximizes $p_h(c)$.
\end{proposition}
\begin{proof}
To avoid carrying the same denominator through every comparison, define the total weight, for any history $h$ and $c\in R$, by
\[
 M_h(c)=\sum_{\substack{w\in S_n:\,(w_1,\ldots,w_k)=h\\ w_{k+1}=c}}f(v(w)).
\]
The sum runs over all permutations with the specified prefix, including those of weight zero. For two adjacent values in $R$, write $c$ for the preferred value in Lemma~\ref{lem:swap} and $d$ for the other value. The swap $\tau$ is a bijection between the two sets of completions, so
\begin{align}
 M_h(c)-M_h(d)
 &=\sum_{\substack{w\in S_n:\,(w_1,\ldots,w_k)=h\\ w_{k+1}=c}}
 \bigl(f(v(w))-f(v(\tau w))\bigr)\ge0.\label{eq:weight-difference}
\end{align}
Each summand is nonnegative because $v(\tau w)\ge v(w)$ and $f$ is nonincreasing.

Adjacent remaining values on one side are also adjacent in $R$. Applying the first case of the lemma successively on $R_+$ gives $M_h(r_1)\ge\cdots\ge M_h(r_s)$. Applying the second case on $R_-$ gives $M_h(u_1)\le\cdots\le M_h(u_t)$. If both sides are nonempty, no remaining value lies between $u_t$ and $r_1$, so the third case gives the comparison between them according to the direction.

The probability of the history $h$ is $Z^{-1}\sum_{d\in R}M_h(d)$. Since this probability is positive, the denominator in
\begin{equation}\label{eq:conditional}
 p_h(c)=\frac{M_h(c)}{\sum_{d\in R}M_h(d)}
\end{equation}
is strictly positive. All the weight comparisons therefore become conditional probability comparisons. The best card on either side is the one nearest to $L$, and the comparison between sides gives precisely $G(h)$. If one side is empty, its counterpart's ordering suffices. For the empty history, only $R_+$ is present, so guessing $1$ first is optimal. If only one card remains, its conditional probability is $1$.
\end{proof}

The conditional probabilities are defined using the entire history, but choosing a maximizer requires only the remaining set, the last revealed card, and the most recent direction. The conditional distribution itself may retain more information about the history; an optimal choice does not require computing it explicitly.

\begin{proof}[Proof of Theorem~\ref{thm:general}]
For each $0\le k<n$, write $H_k=(W_1,\ldots,W_k)$ for the random history, with $H_0$ empty. Fix a history $h$ of positive probability and extend $p_h(c)$ to be zero on revealed cards. A deterministic strategy chooses a card $A_{k+1}(h)$ after this history, so Proposition~\ref{prop:ordering} gives
\[
 \mathbb P_f(A_{k+1}=W_{k+1}\mid H_k=h)
 =p_h(A_{k+1}(h))\le p_h(G(h)).
\]
A randomized strategy can be written as $A_{k+1}=A_{k+1}(H_k,U)$, where $U$ is independent of the entire permutation $W$. Since $H_k$ is a function of $W$, the auxiliary variable $U$ and the next card remain independent conditional on $H_k=h$. More explicitly, for any auxiliary event $\{U\in B\}$,
\[
 \mathbb P_f(U\in B,W_{k+1}=c\mid H_k=h)
 =\mathbb P(U\in B)\,p_h(c).
\]
To see this, apply independence to the event $\{H_k=h,W_{k+1}=c\}$ and divide by $\mathbb P_f(H_k=h)>0$. Thus, if $\alpha_h(c)$ is the probability of guessing $c$ after this history, then
\[
 \mathbb P_f(A_{k+1}=W_{k+1}\mid H_k=h)
 =\sum_{c\in[n]}\alpha_h(c)p_h(c)
 \le p_h(G(h)),
\]
where $\alpha_h(c)\ge0$ and $\sum_c\alpha_h(c)=1$. Deterministic strategies are included in this formula.

Taking expectations over all histories of positive probability yields
\[
 \mathbb P_f(A_{k+1}=W_{k+1})
 \le\mathbb P_f(G(H_k)=W_{k+1}).
\]
Histories of probability zero do not affect either probability. Guesses change neither the actual deck order nor the cards subsequently revealed, so $H_k$ has the same distribution under the two strategies being compared. Summing over $k=0,\ldots,n-1$ gives $\mathbb E_f C_A\le\mathbb E_f C_G$.
\end{proof}
\section{Proof of the shelf-shuffling conjecture}\label{sec:application}
The general theorem is now proved. To obtain the shelf-shuffling result, it remains to show that the probability of the actual output permutation is a nonincreasing function of its valley count. We first specify the correspondence between the shuffling and permutation conventions, then use the probability formula of Diaconis, Fulman, and Holmes and give a short proof of its monotonicity. The two conclusions of the main theorem, for each history and for the total expectation, then follow together.

When applying Theorem~\ref{thm:general}, it is necessary to distinguish the actual output permutation from its inverse. Assign label $2j-1$ to cards placed on top of shelf $j$ and label $2j$ to cards placed on its bottom. The labels are independent and uniform on $\{1,\ldots,2m\}$. Because cards are processed in decreasing order, the top cards on each shelf end up in increasing order, the bottom cards in decreasing order, and, on each shelf, all top cards precede all bottom cards. Thus the output is obtained by sorting labels increasingly, with increasing order within odd-label groups and decreasing order within even-label groups. Empty groups do not change the parity of later labels. This is precisely the output convention in~\cite[Section 3.1, Description 1]{DFH}. Fulman and Petersen~\cite{FP} use $P$-partitions to give a unified account of permutation probabilities for shelf and riffle shuffling, explicitly distinguishing the output permutation from its inverse. All formulas below concern the actual output deck order.

By~\cite[Theorem 3.1]{DFH}, the output probability of any $w\in S_n$ is $q_{n,m}(v(w))$, where
\begin{equation}\label{eq:shelf-weight}
 q_{n,m}(v)=\frac{4^{v+1}}{2(2m)^n}
 [t^{m-v-1}]\frac{(1+t)^{n-1-2v}}{(1-t)^{n+1}},
 \qquad 0\le v\le K.
\end{equation}
Here $[t^r]$ denotes the coefficient of $t^r$ in a formal power series, taken to be zero when $r<0$. The original formula is written as
\[
 \frac{1}{2(2m)^n}[t^m]
 \frac{(1+t)^{n+1}}{(1-t)^{n+1}}
 \left(\frac{4t}{(1+t)^2}\right)^{v+1}.
\]
Factoring out $4^{v+1}t^{v+1}$ leaves the exponent $n-1-2v$ on $(1+t)$ and shifts the coefficient index to $m-v-1$, giving~\eqref{eq:shelf-weight}.

The required monotonicity was proved in~\cite[Corollary 3.3]{DFH}. The following coefficient proof makes its validity for all $n,m$ immediate.

\begin{lemma}\label{lem:weight}
For every $n,m\ge1$, the sequence $q_{n,m}(0),\ldots,q_{n,m}(K)$ is nonincreasing.
\end{lemma}
\begin{proof}
If $K=0$, there are no adjacent terms to compare. Suppose $0\le v<K$. Expressing two adjacent terms under the same coefficient extraction gives
\begin{align*}
 q_{n,m}(v)-q_{n,m}(v+1)
 &=\frac{4^{v+1}}{2(2m)^n}[t^{m-v-1}]
 \frac{(1+t)^{n-3-2v}\bigl((1+t)^2-4t\bigr)}{(1-t)^{n+1}}\\
 &=\frac{4^{v+1}}{2(2m)^n}[t^{m-v-1}]
 \frac{(1+t)^{n-3-2v}}{(1-t)^{n-1}}\ge0.
\end{align*}
We used $(1+t)^2-4t=(1-t)^2$. Since $v<K$, we have $n\ge3$ and $n-3-2v\ge0$, so the numerator is a polynomial with nonnegative coefficients. The reciprocal of the denominator expands as
\[
 (1-t)^{-(n-1)}=\sum_{r\ge0}\binom{n+r-2}{r}t^r,
\]
which also has nonnegative coefficients. If $m-v-1<0$, the extracted coefficient is zero by convention, and the inequality still holds.
\end{proof}

\begin{proof}[Proof of Theorem~\ref{thm:shelf}]
Formula~\eqref{eq:shelf-weight} assigns probability $q_{n,m}(v(w))$ to each output permutation $w$. Nonnegativity of probabilities and Lemma~\ref{lem:weight} show that $q_{n,m}$ is nonnegative and nonincreasing. Taking $f=q_{n,m}$, normalization gives $Z=1$. Theorem~\ref{thm:general} therefore applies and gives~\eqref{eq:main-local} and~\eqref{eq:main-optimality}.
\end{proof}
\section{Conclusions and further questions}\label{sec:conclusion}
We have proved that the direction-tracking strategy is optimal under complete feedback after shelf shuffling for all $n,m\ge1$. The proof gives more than maximal total expectation. After every history of positive probability, the remaining cards on either side of the last revealed card are ordered by conditional probability according to their distance from it; the most recent direction determines the comparison between the nearest cards on opposite sides. The known valley probability formula and our swap comparison together establish the original conjecture. The swap argument itself applies to any distribution with nonincreasing valley weights.

Knowing the strategy turns the subsequent study into the analysis of an explicit random variable. Write the optimal expected reward and the harmonic numbers as
\[
V_{n,m}=\mathbb E_{n,m}C_G,
\qquad H_r=\sum_{j=1}^{r}\frac1j\quad(r\ge1).
\]
The questions below concern, respectively, substantial retained order information, proximity to a uniform permutation, and changes to the shuffling rule. They are not among the results proved here.

\subsection{Reward asymptotics for a fixed number of shelves}
Clay~\cite[Section 6, Conjecture 3]{Clay} proposed that the optimal expectation should be approximately $nH_{2m}/(2m)$ when the ratio of deck size to number of shelves is not too small. The intuition is to view the output as a collection of monotone segments and add the usable order information within each segment. We have proved that the proposed strategy is indeed optimal, but this approximation still requires control of the segment boundaries that the player cannot observe.

\begin{problem}\label{prob:fixed}
For each fixed integer $m\ge2$, is it true that
\[
\lim_{n\to\infty}\frac{V_{n,m}}n=\frac{H_{2m}}{2m}?
\]
If so, what is the order of the error $V_{n,m}-nH_{2m}/(2m)$, and how does it depend on $m$?
\end{problem}

Problem~\ref{prob:fixed} expresses the approximation as a limit and an error question for a fixed number of shelves. The one-shelf reward distribution and limit laws are studied in~\cite{CKT}. Kuba and Panholzer~\cite{KP} studied the limiting distribution of the optimal reward after one riffle shuffle, and Kuba~\cite{KubaBias} analyzed strategies, reward distributions, and changes in limiting behavior as the bias varies in biased riffle shuffling. These results suggest studying fluctuations of the optimal shelf-shuffling reward as well as the leading term of its expectation. Incorrect guesses at boundaries between shelves need not admit the same decomposition as in the one-shelf case. Even an expectation error bound uniform in $m$ would show how quickly the number of shelves may grow with the deck size while the approximation remains valid.

\subsection{When does the guessing reward approach the uniform benchmark?}
For a uniform permutation, all remaining cards are equally likely at every round, so the optimal expectation is exactly $H_n$. Under any distribution, the largest conditional probability when $r$ cards remain is at least $1/r$. Summing over rounds gives $V_{n,m}\ge H_n$. The difference $V_{n,m}-H_n$ thus quantifies order information that remains usable for prediction after shuffling.

\begin{problem}\label{prob:growing}
As $n\to\infty$ and the number of shelves $m=m(n)$ grows with the deck size, under what growth conditions do we have, respectively,
\[
\frac{V_{n,m(n)}}{H_n}\longrightarrow1,
\qquad\text{or, more strongly,}\qquad
V_{n,m(n)}-H_n\longrightarrow0?
\]
Can necessary and sufficient conditions be given for each limit?
\end{problem}

Vanishing relative error and vanishing absolute advantage are different goals. Following~\cite{DFH}, Chen and Ottolini~\cite{CO} determined the total variation cutoff scale for unbiased shelf shuffling: the critical order of the number of shelves is $n^{5/4}$. Clay~\cite{ClayStats} studied central limit theorems for descents and inversions in the same shuffling model. These works describe randomization of the full distribution and of particular permutation statistics, respectively. The guessing reward gives another quantity for comparison, determined by successive revelation of the cards. Specifically, let $\mu_{n,m}$ be the shuffling distribution on $S_n$, let $U_n$ be the uniform distribution, and set
\[
d_{\mathrm{TV}}(\mu,\nu)=\sup_{B\subseteq S_n}|\mu(B)-\nu(B)|.
\]
Since $0\le C_G\le n$ and $G$ always guesses a remaining card, we have $\mathbb E_{U_n}C_G=H_n$, and hence
\[
0\le V_{n,m}-H_n\le n\,d_{\mathrm{TV}}(\mu_{n,m},U_n).
\]
The upper bound follows by comparing the tail probabilities in $\mathbb E C_G=\int_0^n\mathbb P(C_G>t)\,dt$. It does not imply that the two scales coincide: the guessing reward may approach the uniform benchmark before the full distribution is well mixed. Determining whether the bound captures the correct scale requires understanding which prefixes retain a substantial predictive advantage.

\subsection{Biased shuffling and the scope of the strategy}
The unbiased top-or-bottom choice is a key assumption of the valley weight formula. Now let each card still choose a shelf independently and uniformly, but place it on top with probability $p\in(0,1)$ and on the bottom with probability $1-p$, leaving all other rules unchanged. Write $\mathbb E_{n,m,p}$ for expectation in this model. The original strategy $G$ is still defined by~\eqref{eq:strategy}.

\begin{problem}\label{prob:biased}
For given $n\ge3,m\ge2$, determine the parameter set for which the original direction-tracking strategy remains optimal:
\[
\left\{p\in(0,1):
\mathbb E_{n,m,p}C_G
=\max_{A\in\mathcal A_n}\mathbb E_{n,m,p}C_A\right\}.
\]
For the remaining parameters, can an optimal choice still be determined solely by the remaining set, the last revealed card, and the most recent direction?
\end{problem}

We have proved that $p=1/2$ belongs to this set. Clay, Kuba, and Tripathi~\cite{CKT} determined the optimal strategy in the biased one-shelf case, so Problem~\ref{prob:biased} concerns several shelves. For the biased shelf model, Tripathi~\cite{TripathiLaw} gave an exact permutation probability formula, showing that it depends jointly on the numbers of descents and valleys, and studied its mixing scale. Another work of Tripathi~\cite{TripathiStats} studies the distributions of related permutation statistics. Here the number of descents is the number of positions $i$ for which $w_i>w_{i+1}$. Thus an explicit probability formula is available for the extension, but our assumption of nonincreasing weights depending only on valleys no longer applies directly. One concrete approach is to track the effects of swapping adjacent remaining values on both descents and valleys, and determine when the resulting weight comparison preserves the original ordering and when additional history information is needed. This could both yield broader optimality theorems and explain why the present rule fails outside its range of validity.
\clearpage
\appendix
\section*{Appendix}
The main text determines an optimal strategy after unbiased shelf shuffling. Whether optimal choices are unique, whether the player must know the number of shelves, and how to analyze the number of correct guesses deserve separate consideration. The following three appendices explain what else the conditional probability ordering tells us and which aspects of reward analysis it leaves open.

Appendix~\ref{app:ties} examines ties between optimal choices. A four-card example shows that distinct candidates can have the same conditional probability even when permutation weights strictly decrease with the valley count. The swap proof also yields a necessary and sufficient condition for a strict comparison between adjacent candidates. Appendix~\ref{app:mixture} treats an unknown number of shelves. Nonincreasing valley weights remain nonincreasing under mixtures, so the same strategy is optimal for every prior distribution on the number of shelves, without first estimating it. These two conclusions describe, respectively, the flexibility of optimal choices and the strategy's independence from knowledge of the number of shelves.

Appendix~\ref{app:comparison} turns to reward distributions, comparing our setting with complete-feedback guessing in a uniformly shuffled deck with repeated card values. In the latter model, remaining multiplicities determine optimal choices; here the candidate ordering also uses the observed direction. Existing multiplicity models show that identifying an optimal strategy is only the beginning of reward analysis. Together with the equality conditions of Appendix~\ref{app:ties}, the final appendix suggests decomposing shelf-shuffling rewards according to guessing states, connecting this comparison to the fixed-shelf asymptotic question in the main text.
\section{Why optimal choices need not be unique}\label{app:ties}
Direction tracking provides an optimal choice without requiring it to be strictly better than every other choice. Four cards already suffice to exhibit both a strict comparison and a tie for the optimum. We compute these examples and then explain how the swap proof detects equality between adjacent candidates.

Take $n=4,m=2$. By~\eqref{eq:shelf-weight}, the probabilities of permutations with zero and one valley, respectively, are
\[
 q_{4,2}(0)=\frac1{128}[t](1+t)^3(1-t)^{-5}
 =\frac{3+5}{128}=\frac1{16},
 \qquad
 q_{4,2}(1)=\frac1{32}.
\]
After $(1,3)$ has been revealed, there are only two completions: $(1,3,4,2)$ has no valley, and $(1,3,2,4)$ has one. Thus
\[
 p_{(1,3)}(4)=\frac23,\qquad p_{(1,3)}(2)=\frac13.
\]
Direction tracking chooses $4$, and the comparison is strict.

After only $(2)$ has been revealed, the possibilities are
\[
\begin{array}{c|c|c}
\text{Next card}&\text{Completions}&\text{Total probability}\\\hline
1&(2,1,3,4),\ (2,1,4,3)&1/16\\
3&(2,3,1,4),\ (2,3,4,1)&3/32\\
4&(2,4,1,3),\ (2,4,3,1)&3/32
\end{array}
\]
The three weights sum to $1/4$, so
\[
 p_{(2)}(1)=\frac14,\qquad p_{(2)}(3)=p_{(2)}(4)=\frac38.
\]
The strategy $G$ chooses $3$. Changing its choice to $4$ at this history, and retaining $G$ at every other history, gives another optimal strategy. Guesses do not affect subsequent histories, and the conditional success probability in this round is unchanged, so the modification preserves the total expectation.

More generally, the swap also determines whether the comparison between adjacent remaining values is strict. Fix a history $h$ of positive probability, let $c$ be the preferred value and $d$ the other value as in Lemma~\ref{lem:swap}, and define
\[
 \mathcal E_h(c,d)=\{w\in S_n:(w_1,\ldots,w_k)=h,\ w_{k+1}=c,
 \ v(\tau w)=v(w)+1\}.
\]
This is the set of completions for which the swap adds exactly one valley. Removing the terms in~\eqref{eq:weight-difference} whose valley counts do not change gives
\begin{equation}\label{eq:strictness}
 M_h(c)-M_h(d)
 =\sum_{w\in\mathcal E_h(c,d)}\bigl(f(v(w))-f(v(w)+1)\bigr).
\end{equation}
Every summand is nonnegative, and $v(w)+1\le K$, so every weight is defined. Together with~\eqref{eq:conditional}, this shows that $p_h(c)>p_h(d)$ if and only if some $w\in\mathcal E_h(c,d)$ satisfies $f(v(w))>f(v(w)+1)$. If there is no such permutation, the two candidates tie. Thus, strict decrease of the weights with the valley count alone does not imply uniqueness of the optimal strategy: whether the swap actually adds a valley matters as well.
\section{The same strategy when the number of shelves is unknown}\label{app:mixture}
The number of shelves affects the conditional probability of the next card but does not appear in the direction-tracking rule. Thus a player who does not know how many shelves the machine uses can still follow the same optimal strategy. We now allow the number of shelves itself to be random and explain why it need not be estimated first.

Fix $n\ge1$. Let a positive integer $M$ be chosen at random before the shuffle, with $\pi_m=\mathbb P(M=m)$, where $\pi_m\ge0$ and $\sum_{m\ge1}\pi_m=1$. Conditional on $M=m$, independently perform one $m$-shelf shuffle according to Section~\ref{sec:model}. The player observes the actual cards but not $M$.

\begin{proposition}\label{prop:unknown-shelves}
For every distribution $(\pi_m)_{m\ge1}$ on the number of shelves, the direction-tracking strategy $G$ is optimal in this complete-feedback guessing problem.
\end{proposition}
\begin{proof}
For every $v\in\{0,\ldots,K\}$, set
\[
 f(v)=\sum_{m\ge1}\pi_m q_{n,m}(v).
\]
Every admissible $v$ is realized by a permutation. For example, arrange the first $2v+1$ values as $(2,1,4,3,\ldots,2v,2v-1,2v+1)$ and append the remaining values in increasing order. This gives exactly $v$ valleys; for $v=0$, take the increasing permutation. Hence $q_{n,m}(v)$ is the probability of an individual permutation and lies between $0$ and $1$, so the series converges. Each $q_{n,m}$ is nonincreasing, making $f$ nonnegative and nonincreasing. By the law of total probability,
\[
 \mathbb P(W=w)=\sum_{m\ge1}\pi_m q_{n,m}(v(w))=f(v(w)).
\]
Summing over the finite set $S_n$ gives $\sum_w f(v(w))=1$. All hypotheses of Theorem~\ref{thm:general} hold, and the conclusion follows.
\end{proof}

The conditional distribution of the number of shelves will generally change after a history is observed; the proposition does not treat it as the original prior. Instead, it applies the general theorem to the unconditional output distribution, and that theorem already guarantees optimality after every history of positive probability. Taking $\pi_m$ concentrated at any one shelf count also shows that, for a fixed but unknown number of shelves, $G$ is optimal simultaneously in every model. This conclusion concerns a single shuffle after the number of shelves has been chosen.
\section{From remaining multiplicities to remaining order}\label{app:comparison}
Complete-feedback guessing in a uniformly shuffled deck provides a useful comparison. When card values may repeat, optimal choices are determined by remaining multiplicities. In our shelf model, cards are distinct, and the remaining values together with the most recently observed direction order the candidates. This difference also explains why the reward distribution requires further study even after an optimal strategy has been identified.

Consider a uniformly shuffled deck with $r$ card types, where type $i$ occurs $s_i$ times and the total number of cards is $N=\sum_{i=1}^r s_i$. When $t$ cards remain, let $R_i(t)$ be the remaining number of cards of type $i$. Conditional on the revealed history, the probability that the next card has type $i$ is $R_i(t)/t$. Guesses do not affect later revelations, so choosing a type of largest remaining multiplicity at every round is optimal. This is the complete-feedback model studied by Diaconis and Graham~\cite[Section 2]{DG}. If $S$ denotes the total number of correct guesses under this strategy, then
\[
 \mathbb E S=\sum_{t=1}^{N}\frac{\mathbb E\max_{1\le i\le r}R_i(t)}{t}.
\]
The formula follows by adding the conditional expectations over all rounds. When every card type occurs only once, the largest remaining multiplicity is always $1$, recovering our uniform benchmark $H_N$.

Even with such a simple optimal strategy, the reward asymptotics vary with the deck composition. He and Ottolini~\cite{HO} studied the optimal expectation when the maximum multiplicity stays bounded and the number of types grows, allowing different types to have different multiplicities. Ottolini and Steinerberger~\cite{OS} considered parameter ranges in which both the number of types and their multiplicities grow. For fixed multiplicity, Ottolini and Tripathi~\cite{OT} established a central limit theorem for the reward with an error bound for normal approximation. These results cannot be applied here merely by replacing multiplicity with the number of shelves: the conditional distribution in the uniform deck is determined by remaining counts, whereas shelf shuffling also retains order constraints.

Ties between optimal choices can also guide the analysis of the reward distribution. In the two-type model, when the remaining counts are equal, either choice has success probability $1/2$; when only one type remains, a correct guess is certain. Kuba and Panholzer~\cite{KPColors} used these distinct states to decompose the total reward and studied the resulting joint distributions and limits. Appendix~\ref{app:ties} shows that, in shelf shuffling, ties can also arise when two sets of completions have equal weights under the swap, a condition that cannot be recognized from remaining counts alone. Decomposing $C_G$ according to these equality cases and deriving reward recurrences could connect our conditional probability ordering to the study of fluctuations associated with Problem~\ref{prob:fixed}.
\clearpage
\begingroup
\raggedright

\endgroup
\end{document}